\documentclass[12pt]{article}
\usepackage{amssymb}
\usepackage{amsfonts}
\usepackage{amsmath}
\usepackage[usenames]{color}
\usepackage{mathrsfs}
\usepackage{amsfonts}
\usepackage{amssymb,amsmath}
\usepackage{CJK}
\usepackage{cite}
\usepackage{cases}
\usepackage{amsthm}
\usepackage{amssymb}
\usepackage{graphicx}
\usepackage{amsmath}
\usepackage[all]{xy}
\usepackage{enumerate}
\usepackage{amscd}
\usepackage{cases}

\def\cl{\centerline}

\def\G{\mathcal{G}}
\def\b{\beta}
\def\vs{\vspace*}

\def\L{\mathcal{L}}
\def\T{\mathcal{T}}
\def\V{\mathcal{V}}

\def\M{\mathcal{M}}
\def\Z{\mathbb{Z}}
\def\N{\mathbb{N}}
\def\b{\mathfrak{b}}

\def\G{\mathcal{G}}

\def\C{\mathbb{C}}

\def\p{\mathfrak{p}}

\def\ni{\noindent}

\numberwithin{equation}{section}
\newtheorem{theo}{Theorem}[section]
\newtheorem{defi}[theo]{Definition}
\newtheorem{coro}[theo]{Corollary}
\newtheorem{lemm}[theo]{Lemma}
\newtheorem{prop}[theo]{Proposition}
\newtheorem{clai}{Claim}

\newtheorem{rema}[theo]{Remark}

\newtheorem{remark}[theo]{Remark}

\begin{document}
	\begin{center}
		\cl{\large\bf \vs{8pt}  Restricted representations  of the twisted }
		\cl{\large\bf \vs{8pt}    $N=2$    superconformal algebra}
		\cl{ Haibo Chen, Yucai Su and Yukun Xiao}
	\end{center}
	
	{\small
		\parskip .005 truein
		\baselineskip 3pt \lineskip 3pt
		
		\noindent{{\bf Abstract:}
			In this paper, we  construct a large class of new simple modules over the twisted $N=2$ superconformal algebra. These new simple    modules
			are  restricted   modules constructed from the simple
			modules over certain finite-dimensional solvable Lie superalgebras,  including  various versions of Whittaker modules.  We elaborate that they are also  twisted modules for the universal $N=2$ superconformal vertex (operator) algebra.   On the other hand, we give an explicit characterization of the simple restricted modules over the twisted $N=2$ superconformal algebra $\mathcal{T}$   on which $T_t$ in $\mathcal{T}$ acts injectively for some $t\in \frac{1}{2}+\Z_+$.
			\vs{5pt}
			
			\ni{\bf Key words:}
			Twisted  $N=2$  superconformal algebra,  restricted module,  twisted module.}
		
		\ni{\it Mathematics Subject Classification (2020):} 17B10, 17B65, 17B68, 17B69.}
	\parskip .001 truein\baselineskip 6pt \lineskip 6pt
	\section{Introduction}
	The $N=2$ superconformal algebras were constructed independently by Kac  and    Ademollo et
	al. in  \cite{K1,Ael},
	with four sectors,  called the Neveu-Schwarz sector,
	the Ramond sector, the topological sector  and the twisted sector.
	In fact, the first three untwisted sectors are isomorphic to each other by ``spectral flows" and  ``topological twists" (see \cite{DVV, SS}). Essentially, there are  two classes of different
	sectors which we called untwisted and twisted $N=2$ superconfromal algebra.
	
	The   representation theory  of the $N=2$ superconformal algebras has  been   extensively investigated  by mathematicians.  The weight modules  over the untwisted  $N = 2$ superconformal algebra have been studied in many papers such as \cite{D,FJS,I,LSZ,ST}. Recently, some non-weight modules related to the $U(\mathfrak{h})$-free modules, Whittaker modules and $D$-modules over the   untwisted $N=2$ superconformal algebras were studied in \cite{YYX2,LLD,CLP}. As for the twisted $N=2$ superconformal algebra, its $U(\mathfrak{h})$-free modules, Verma modules  and  indecomposable modules of the intermediate series were respectively investigated  in \cite{CDL2,IK,DG,LSZ}.
	
	In this paper, we mainly care about the restricted modules over the twisted $N=2$ superconformal algebra with a vertex algebraic motivation.  As we know,  there are close connections between the representations of  Lie algebras and those of related vertex (operator) algebras with the same central charge (see \cite{LL,L,L2}). For classical examples, the restricted modules over the Virasoro algebra are exactly the weak modules over the Virasoro vertex operator algebras;  the restricted modules over the twisted affine Lie algebras are identical to the weak twisted module over  a vertex operator algebra associated to the corresponding untwisted affine Lie algebras; the restricted modules for the Ramond algebra are  weak twisted modules for the Neveu-Schwarz vertex operator superalgebra.  In particular, for each central charge $c$, there is a vertex (operator) superalgebra associated to the untwisted $N=2$ superconformal algebra, called the universal $N=2$ superconformal vertex (operator) algebra (see \cite{A}). We will illustrate the fact that the  (weak) $\psi$-twisted modules for universal   $N=2$ superconformal vertex (operator) algebra can be identified with the restricted modules for the twisted $N=2$ superconformal algebra (see Proposition \ref{VOA-version}), where $\psi$ is a specific automorphism of order 2. That is why we  study the twisted $N=2$ superconformal algebra and its restricted representations.

	Whittaker modules  as an important class of restricted modules, were first   introduced  for the finite dimensional Lie algebra $\mathfrak{sl}(2)$      in \cite{AP}.
	At the same time, Whittaker modules over the finite dimensional complex semisimple Lie algebras were defined in \cite{K1}.  After that, they were widely studied over some other objects, such as affine Lie algebra, infinite dimensional Lie (super)algebras and so on (see,  e.g. \cite{ALZ,C,LWZ,LLD,LZ3,LGZ}). On the other hand, Whittaker  modules are also investigated in the framework of vertex algebra theory in \cite{ALZ,ALTY1,ALTY}, which could provide irreducible modules for  orbifolds for corresponding vertex algebras.

	Highest weight modules and Whittaker modules for the  Virasoro algebra  share some same property: the actions of elements in the positive part   are locally finite (see \cite{MZ1}).
	In \cite{MZ}, Mazorchuk and Zhao  gave    a generalized construction of simple Virasoro modules, which  include highest weight modules, Whittaker modules and high order
	Whittaker modules.
	Moreover, the simple restricted modules   over some other infinite dimensional Lie (super)algebras have been investigated (see, e.g. \cite{CG,G,GG,CHS,LPX55,LPXZ,CGHS,TYZ}).
	The fact is,   for the super case,  there are a lot of   differences and new features (see, e.g. \cite{LS,LPX55,S,C}), which we are mainly interested in.
	So in the present paper, we will obtain and characterize a large class of  new simple non-weight  (restricted) modules   over the twisted  $N=2$ superconformal algebra.

	The rest of the present paper is organized as follows.
	In Section $2$, we  recall  some definitions  and propositions related to Lie superalgebra and vertex operator superalgebra. In Section $3$, two non-isomorphic sectors of the $N=2$ superconformal algebra are presented. We show that the restricted modules for the twisted $N=2$ superconformal algebra correspond to the  weak $\psi$-twisted modules for the universal $N=2$ superconformal vertex operator superalgebra.
	In Section $4$, a class of new simple  restricted modules over the  twisted  $N=2$  superconformal algebra $\T$  are constructed (see Theorem \ref{th1}).
	In Section $5$, we present a characterization of simple restricted $\T$-modules under certain conditions  in Theorem \ref{th2}, which reduces the problem of classification of simple restricted
	$\T$-modules to classification of simple modules over a class of finite-dimensional solvable Lie superalgebras.
	In Section $6$, we show  some examples of restricted $\T$-modules such as generalized Whittaker
	modules   and high order Whittaker modules.

	Throughout this paper,  we denote by $\C$,  $\Z$, $\N$ and $\Z_+$ the sets of complex numbers,  integers,
	nonnegative integers and positive integers, respectively.
	Note that $\frac{1}{2}\Z=\Z\bigcup(\frac{1}{2}+\Z)$. All vector superspaces
	(resp. superalgebras, supermodules)   and
	spaces (resp. algebras, modules)    are considered   over
	$\C$. We use $U(\G)$ to denote
	the universal enveloping algebra for a Lie (super)algebra $\G$.
	We also denote by
	$\delta_{i,j}$ the Kronecker symbol and $\mathfrak{i}=\sqrt{-1}$ the imaginary unit.

	\section{Preliminaries}
	In this section,   some notations and definitions related  to Lie superalgebras, vertex superalgebras and their representations are presented (c.f. \cite{CW,K1, L, L2}).

	Assume that  $W=W_{\bar0}\oplus W_{\bar1}$ is a $\Z_2$-graded vector space. Any element $w\in W_{\bar0}$
	is said to be  {\em even}  and any element $w\in W_{\bar1}$ is said to be  {\em odd}. Set $|w|=0$ if
	$w\in W_{\bar0}$ and $|w|=1$ if $w\in W_{\bar1}$. All elements  in  $W_{\bar0}$ or  $W_{\bar1}$ are called  {\em homogeneous}.
	Throughout this paper, all elements in Lie superalgebras and modules are homogenous,   all modules for Lie superalgebras    are $\Z_2$-graded and
	all simple modules are non-trivial  unless
	specified.
	
	Let $\mathcal{L}$ be a Lie superalgebra. An {\em $\mathcal{L}$-module} is a $\Z_2$-graded vector space $W$ together
	with a bilinear map $\mathcal{L}\times W\rightarrow W$, denoted $(x,w)\mapsto xw$ such that
	$$x(yw)-(-1)^{|x||y|}y(xw)=[x,y]w\quad
	\mathrm{and}\quad
	\mathcal{L}_{\bar i} W_{\bar j}\subseteq  W_{\bar i+\bar j},$$
	for all $x, y \in \mathcal{L}, w \in W$. Thus there is a parity-change functor $\Pi$ on the category of
	$\mathcal{L}$-modules to itself. That is to say, for any $\mathcal{L}$-module
	$W=W_{\bar0}\oplus W_{\bar1}$, we have a new module $\Pi(W)$ with the same underlining space with the
	parity exchanged, namely, $\Pi(W_{\bar 0})=W_{\bar 1}$ and $\Pi(W_{\bar 1})=W_{\bar0}$.
	\begin{defi}\rm
		Assume that  $P$  is a module of  $\L$ and $x\in\L$.
		
		\begin{itemize}
			\item[{\rm (i)}] If for any $v\in P$ there exists $m\in\Z_+$ such that $x^mv=0$,  then   the action of  $x$  on $P$ is called {\em locally nilpotent}.
			Similarly, the action of $\L$   on $P$  is called {\em locally nilpotent} if for any $v\in P$ there exists $m\in\Z_+$ such that $\L^mv=0$.
			\item[{\rm (ii)}]  If for any $v\in P$, we have  $\mathrm{dim}(\sum_{m\in\Z_+}\C x^mv)<+\infty$, then  the action of $x$ on $P$ is called {\em locally finite}.
			Similarly, the action of $\L$  on $P$  is called  {\em locally finite}
			if for any $v\in P$ we get $\mathrm{dim}(\sum_{m\in\Z_+} {\L}^mv)<+\infty$.
		\end{itemize}
	\end{defi}
	Clearly,   the action of $x$ on $P$ being locally nilpotent, implies that the action of $x$ on $P$ is locally finite.
	If $\L$ is a finitely generated Lie superalgebra,   the action of
	$\L$
	on $P$ is locally nilpotent, which implies that  the action of $\L$ on $P$ is locally finite.
	
	Now we turn our attention to  vertex (operator) superalgebras.
	\begin{defi}
		A {\it vertex superalgebra} is a quadruple
		$(V,{\bf 1},D,Y)$, where $V=V^{(0)}\oplus V^{(1)}$ is a $\Z_{2}$-graded
		vector space, $D$ is a $\Z_{2}$-endomorphism
		of $V$, ${\bf 1}$ is a specified element called the {\it vacuum} of $V$,
		and $Y$ is a linear map
		\begin{eqnarray}
			Y(\cdot,z):& &V\rightarrow ({\rm End}V)[[z,z^{-1}]];\nonumber\\
			& &a\mapsto Y(a,z)=\sum_{n\in \Z}a_{n}z^{-n-1}\;\;(\mbox{where }
			a_{n}\in {\rm End}V)
		\end{eqnarray}
		such that
		\begin{eqnarray*}
			{\rm (V1)}& &\mbox{For any }a,b\in V, a_{n}b=0\;\;\;\mbox{ for }n
			\mbox{ sufficiently large;}\\
			{\rm (V2)}& &[D,Y(a,z)]=Y(D(a),z)={d\over dz}Y(a,z)\;\;\mbox{  for any }a\in V;\\
			{\rm (V3)}& &Y({\bf 1},z)=Id_{V}\;\;\;\mbox{(the identity operator of $V$)};\\
			{\rm (V4)}& &Y(a,z){\bf 1}\in ({\rm End}V)[[z]] \mbox{ and }\lim_{z \rightarrow
				0}Y(a,z){\bf 1}=a\;\;\mbox{  for any }a\in V;\\
			{\rm (V5)}& &\mbox{For }\Bbb{Z}_{2}\mbox{-homogeneous }a,b\in V,
			\mbox{ the following {\it Jacobi identity} holds:}
		\end{eqnarray*}
		$$
		\begin{aligned}
			& z_0^{-1} \delta\left(\frac{z_1-z_2}{z_0}\right) Y\left(a, z_1\right) Y\left(b, z_2\right)- (-1)^{|a||b|}z_0^{-1} \delta\left(\frac{z_2-z_1}{-z_0}\right) Y\left(b, z_2\right) Y\left(a, z_1\right) \\
			= & z_2^{-1} \delta\left(\frac{z_1-z_0}{z_2}\right) Y\left(Y\left(a, z_0\right) b, z_2\right) .
		\end{aligned}
		$$
		A vertex superalgebra $V$ is called a vertex operator superalgebra if there is a Virasoro element  $\omega$ of $V$ such that
		\begin{eqnarray*}	
			&{\rm (V6)}&[L(m),L(n)]=(m-n)L(m+n)+\frac{m^{3}-m}{12}\delta_{m+n,0}({\rm rank}V)\\
			& &\mbox{for }m,n\in \Bbb{Z},\mbox{ where }
			Y(\omega,z)=\sum_{n\in \Bbb{Z}}L(n)z^{-n-2},\;{\rm
				rank}V\in \Bbb{C};\\
			&{\rm (V7)}&L(-1)=D,\; i.e., Y(L(-1)a,z)={d\over dz}Y(a,z)\;\;\;\mbox{ for any
			}a\in V;\\
			&{\rm (V8)}&V\!=\!\oplus_{n\in {1\over 2}\Bbb{Z}}
			V_{(n)}\mbox{ is }{1\over 2}\Bbb{Z}\mbox{-graded such that
			}L(0)|_{V_{(n)}}\!=\!nId_{V_{(n)}},\dim V_{(n)} \!<\!\infty,\;\;\;\;\\
			& &\mbox{ and } V_{(n)}=0\;\;\mbox{ for }n\mbox{ sufficiently small. }
		\end{eqnarray*}
	\end{defi}
	
	For two vertex operator superalgebras $V_{1}$ and $V_{2}$, a {\it homomorphism} from $V_{1}$ to $V_{2}$ is required to map the Virasoro element of $V_{1}$ to that of
	$V_{2}$.
	An {\it automorphism} of a vertex (operator) superalgebra $V$ is
	1-1 onto homomorphism from $V$ to $V$.
	Let $\sigma$ be an automorphism of order $T$ of a  vertex
	superalgebra $V$.
	Then,
	\begin{eqnarray}
		V=V^{0}\oplus V^{1}\oplus \cdots \oplus V^{T-1}\end{eqnarray}
	where $V^{k}$ is the eigenspace of $V$ for $\sigma$ with eigenvalue
	${\rm exp}\left({2k\pi\mathfrak{i}\over T}\right)$.
	
	Let $V$ be any vertex superalgebra. Then the linear map $\sigma: V \rightarrow V ; \sigma(a+b)=$ $a-b$ for $a \in V^{(0)}, b \in V^{(1)}$, is an automorphism of $V$, which was called the {\it canonical automorphism} (c.f. \cite{FFR}).

	\begin{defi}
		Let $(V,{\bf 1},D,Y)$ be a
		vertex superalgebra with an automorphism $\sigma$
		of order $T$. A $\sigma$-twisted $V$-module is a triple $(M,d,
		Y_{M})$ consisting of a super vector
		space $M$, a $\Z_{2}$-endomorphism $d$ of $M$ and a linear map
		$Y_{M}(\cdot,z)$ from $V$ to
		$({\rm End}M)[[z^{{1\over T}},z^{-{1\over T}}]]$ satisfying the
		following conditions:
		
		{\rm(M1)}$\;\;\;\;$For any $a\in V,u\in M, a_{n}u=0$ for $n\in {1\over
			T}\Z$ sufficiently large;
		
		{\rm(M2)}$\;\;\;\;Y_{M}({\bf 1},z)=Id_{M}$;
		
		{\rm(M3)}$\;\;\;\;[d,Y_{M}(a,z)]=Y_{M}(D(a),z)=\displaystyle{{d\over
				dz}Y_{M}(a,z)}$ for any $a\in V$;
		
		{\rm(M4)}$\;\;\;$For any $\Z_{2}$-homogeneous $a\in V^k,b\in V$, the following
		$\sigma$-twisted
		Jacobi identity holds:
		$$
		\begin{aligned}
			& z_0^{-1} \delta\left(\frac{z_1-z_2}{z_0}\right) Y_M\left(a, z_1\right) Y_M\left(b, z_2\right)-\varepsilon_{a, b} z_0^{-1} \delta\left(\frac{z_2-z_1}{-z_0}\right) Y_M\left(b, z_2\right) Y_M\left(a, z_1\right) \\
			= & z_2^{-1} \delta\left(\frac{z_1-z_0}{z_2}\right)\left(\frac{z_1-z_0}{z_2}\right)^{-\frac{k}{T}} Y_M\left(Y\left(a, z_0\right) b, z_2\right).
		\end{aligned}
		$$
		If $V$ is a vertex operator superalgebra, a $\sigma$-twisted
		$V$-module for $V$ as a  vertex
		superalgebra is called a  $\sigma$-twisted {\it weak} module
		for $V$ as a vertex
		operator superalgebra.
	\end{defi}
	Taking $\sigma=\operatorname{id}_V$ gives the definition of classical (weak) $V$-module. It is clear that $V^{0}$ is a vertex subsuperalgebra of $V$ and all $V^{k}$
	$(k=0,1,\ldots,T-1)$ are $V^{0}$-modules.

	The next lemma follows form \cite[Lemma 2.11]{L2}.
	\begin{lemm}\label{2.4}
		Let $V$ be a vertex superalgebra with an
		automorphism $\phi$ of order $T$ and let $(M,Y_{M})$ be a faithful $\phi$-twisted
		$V$-module. Let $\Z_{T}$-homogeneous elements $a\in V^k,b,u^{0},\ldots,u^{n}$ be
		of $V$. Then
		\begin{eqnarray*}
			[Y_{V}(a,z_{1}),Y_{V}(b,z_{2})]
			=\sum_{j=0}^{\infty}{1\over j!}\left(\left({\partial\over\partial
				z_{2}}\right)^{j}z_{1}^{-1}\delta\left({z_{2}\over
				z_{1}}\right)\right)Y_{V}(u^{j},z_{2})
		\end{eqnarray*}
		if and only if
		\begin{eqnarray*}
			[Y_{M}(a,z_{1}),Y_{M}(b,z_{2})]
			=\sum_{j=0}^{\infty}{1\over j!}\left(\left({\partial\over\partial
				z_{2}}\right)^{j}z_{1}^{-1}\delta\!\left({z_{2}\over
				z_{1}}\right)\!\left({z_{2}\over z_{1}}\right)^{{k\over
					T}}\right)Y_{M}(u^{j},z_{2})
		\end{eqnarray*}
		if and only if $a_{i}b=u^{j}$ for $0\le j\le n$, $a_{i}b=0$ for $j>n$.
	\end{lemm}
	\section{$N=2$ superconformal algebra}
	In this section, the definition of the $N=2$ superconformal algebras are presented. We will elaborate the fact that any restricted module for the twisted $N=2$ superconformal algebra is a (weak) twisted module for a vertex (operator) superalgebra associated to the untwisted $N=2$ superconformal algebra. And this is exactly our vertex algebraic motivation to study the twisted $N=2$ superconformal algebra and its restricted modules.
	\subsection{Untwisted $N=2$ superconformal algebra}
	The {\it untwisted} $N=2$ {\it superconformal algebra} $\mathcal{A}$  is the infinite-dimensional Lie superalgebra with basis ${L}_m, {J}_n, {G}_{p}^{ \pm}, C$ with relations given by
	$$
	\begin{aligned}
		& \left[L_m, L_n\right]=(m-n) L_{m+n}+\frac{1}{12}\left(m^3-m\right) \delta_{m+n, 0} C,\\
		& \left[L_m, J_n\right]=-n J_{n+m}},\ {\left[J_m, J_n\right]=\frac{1}{3} m \delta_{m+n, 0}C, \\
		& \left[L_m, G_p^{ \pm}\right]=\left(\frac{1}{2} m-p\right) G_{m+p}^{ \pm}},\ {\left[J_m, G_p^{ \pm}\right]= \pm G_{m+p}^{ \pm}, \\
		& \left[G_p^{+}, G_q^{-}\right]=2 L_{p+q}+(p-q) J_{p+q}+\frac{1}{3}\left(p^2-\frac{1}{4}\right) \delta_{p+q, 0}C, \\
		& \left[G_p^{+}, G_q^{+}\right]=\left[G_p^{-}, G_q^{-}\right]=0,
	\end{aligned}
	$$
	for all $m, n \in \Z$ and $p, q \in \frac{1}{2}+\Z$, where $C$ is the central element.
	By definition,  the even space $\mathcal{A}_{\bar{0}}=\operatorname{span}\{L_m,J_r,C| m,r\in \Z\}$ and  the odd space $\mathcal{A}_{\bar{1}}=\operatorname{span}\{G^{\pm}_{s+\frac{1}{2}}|s\in \Z\}$. And $\mathcal{A}=\oplus_{n\in\frac{1}{2}\Z}\mathcal{A}_{n}$ by defining
	$$\operatorname{deg}L_m=m,\  \operatorname{deg}J_r=r,\ \ \operatorname{deg}G^{\pm}_{s+\frac{1}{2}}=s+\frac{1}{2},\ \operatorname{deg}C=0,$$
	for $m,r,s\in\Z$. Then we have a triangular decomposition $\mathcal{A}=\mathcal{A}_+\oplus\mathcal{A}_{0}\oplus\mathcal{A}_-$ where
	$$\mathcal{A}_{\pm}=\sum_{n=1}^{\infty} \left(\C L_{\pm n}+\C J_{\pm n}+ \C G^{+}_{\pm n \mp\frac{1}{2}}+ \C G^{-}_{\pm n \mp\frac{1}{2}}\right),\quad \mathcal{A}_{0}=\C L_0\oplus\C J_0 \oplus \C C.$$

	Now we recall the definition of Whittaker modules over $\mathcal{A}$ in \cite{LLD}.
	\begin{defi}
		Let  $$\mathcal{K}=\sum_{n=1}^{\infty} \left(\C L_{n}+\C J_{n}+ \C G^{+}_{n+\frac{1}{2}}+ \C G^{-}_{n+\frac{1}{2}}\right),$$  and $\Phi$ be a Lie superalgebra homomorphism $\Phi:\mathcal{K}\rightarrow\mathbb{C}$. An $\mathcal{A}$-module $P$ is called a Whittaker module of type $(\mathcal{K},\Phi)$ if it is generated
		by a vector $w$, which is called a Whittaker vector, with $xw =\Phi(x)w$ and $Cw=\alpha w, J_0w =\beta w$ for any $x\in\mathcal{K}$ and some $\alpha,\beta \in\mathbb{C}$.
	\end{defi}
	Note that a Whittaker module over  $\mathcal{A}$ is simple if and only if $\Phi(L_2)\neq0$
	(see \cite{LLD}). Moreover, the simplicity of the  quotient modules of the  non-simple  universal Whittaker module
	was also studied.

	Denote by $M(c, 0,0)$ the Verma module over $\mathcal{A}$  of central charge $c\in\C$  generated by a lowest weight vector ${\bf 1}$ with $L_0$ and $J_0$ act as $0$. An element $\mathrm{v}\in M(c, 0,0)$ is called a \emph{singular vector} if $\mathrm{v}$ is an eigenvector of $L_0$ and $J_0$ and
	$$L_m\mathrm{v}=J_r\mathrm{v}=G^{\pm}_{q+\frac{1}{2}}\mathrm{v}=0$$
	for $m,r\in\Z_+$ and $q\in \N$. In fact, $ G^{+}_{-\frac{1}{2}} \mathbf{1}$ and $G^{-}_{-\frac{1}{2}} \mathbf{1}$ are two singular vectors here.
	Now set
	$$\bar{M}(c, 0,0)=\frac{M(c, 0,0)}{\left\langle G^{+}_{-\frac{1}{2}} \mathbf{1}\right\rangle+\left\langle G^{-}_{-\frac{1}{2}} \mathbf{1}\right\rangle} $$
	where  $\left\langle G^{\pm}_{-\frac{1}{2}} \mathbf{1}\right\rangle$ is the submodule generated by $G^{\pm}_{-\frac{1}{2}} \mathbf{1}$. Then $\bar{M}(c, 0,0)$ is also a lowest  weight $\mathcal{A}$-module. Still use ${\bf 1}$ denoting the lowest weight vector. Define

$$L(z)=\sum_{m\in \Z}L_mz^{-m-2},\ J(z)=\sum_{r \in\Z}J_rz^{-r-1},\
	G^{\pm}(z)=\sum_{p\in\Z}G^{\pm}_{p+\frac{1}{2}}z^{-p-{2}}. $$

	Then we can see that $L(z)$, $J(z)$, $G^{+}(z)$, $G^{-}(z)$ are mutually local to each other (see also (\ref{untwisted rela})). Using Li's local system theory (c.f. \cite{LL,L}), we have the following result (see \cite{A}).
	\begin{prop}
		$\bar{M}(c, 0,0)$ is  a vertex operator superalgebra, where $$Y(L_{-2}{\bf 1},z)=L(z),\ Y(J_{-1}{\bf 1},z)=J(z),\ Y(G^{\pm}_{-\frac{3}{2}}{\bf 1},z)=G^{\pm}(z).$$  It is called the universal $N=2$ superconformal vertex (operator) algebra.
	\end{prop}
	
	Let $\psi$ be the Lie superalgebra automorphism defined by $L_m\mapsto L_m, J_n\mapsto -J_n, G_p^{\pm}\mapsto G_p^{\mp}$. Then it gives rise to an automorphism of 	$\bar{M}(c, 0,0)$ of order 2 and we still use $\psi$ to denote it. Note that it is not the canonical automorphism $\sigma$ mentioned in Section 2.
	
	For convenience of correspondence,  we consider the substitutions (c.f. \cite{Barr}):
	$$
	G_{n+\frac{1}{2}}^{(1)}=\frac{1}{\sqrt{2}}\left(G_{n+\frac{1}{2}}^{+}+G_{n+\frac{1}{2}}^{-}\right), \quad G_{n+\frac{1}{2}}^{(2)}=\frac{\mathfrak{i}}{\sqrt{2}}\left(G_{n+\frac{1}{2}}^{+}-G_{n+\frac{1}{2}}^{-}\right) .
	$$
	Note then that $G_{n+\frac{1}{2}}^{ \pm}=\frac{1}{\sqrt{2}}\left(G_{n+\frac{1}{2}}^{(1)} \mp \mathfrak{i} G_{n+\frac{1}{2}}^{(2)}\right)$. Under these substitutions, the $N=2$ superconformal algebra has a basis consisting of the central element $C$ and $L_n, J_n, G_{n+\frac{1}{2}}^{(j)}$, satisfying the following relations:
	\begin{equation}\label{N=1}
		\aligned
		& {\left[L_m, L_n\right]=(m-n) L_{m+n}+\frac{1}{12}\left(m^3-m\right) \delta_{m+n, 0}} C,\\
		& {\left[L_m, J_n\right]=-n J_{n+m}},\ {\left[J_m, J_n\right]=\frac{1}{3} m \delta_{m+n, 0}}C, \\
		&{\left[L_m, G_{n+\frac{1}{2}}^{(j)}\right]=\left(\frac{m}{2}-n-\frac{1}{2}\right) G_{m+n+\frac{1}{2}}^{(j)},} \\
		&{\left[J_m, G_{n+\frac{1}{2}}^{(1)}\right]=-\mathfrak{i} G_{m+n+\frac{1}{2}}^{(2)}, \quad\left[J_m, G_{n+\frac{1}{2}}^{(2)}\right]=\mathfrak{i} G_{m+n+\frac{1}{2}}^{(1)},} \\
		&{\left[G_{m+\frac{1}{2}}^{(j)}, G_{n-\frac{1}{2}}^{(j)}\right]=2 L_{m+n}+\frac{1}{3}\left(m^2+m\right) \delta_{m+n, 0} C,} \\
		&{\left[G_{m+\frac{1}{2}}^{(1)}, G_{n-\frac{1}{2}}^{(2)}\right]=-\mathfrak{i}(m-n+1) J_{m+n},}
		\endaligned
	\end{equation}
	for $m, n \in \mathbb{Z}$ and $j=1,2$. After substitutions,
	\begin{equation}\label{subs}
		\aligned
		&G^{(1)}(z)=Y(G^{(1)}_{-\frac{3}{2}}{\bf 1},z)=\frac{1}{\sqrt{2}}\left(G^{+}(z)+G^{-}(z)\right)=\sum_{p\in\Z}G^{(1)}_{p+\frac{1}{2}}z^{-p-{2}}, \\
		&G^{(2)}(z)=Y(G^{(2)}_{-\frac{3}{2}}{\bf 1},z)=\frac{\mathfrak{i}}{\sqrt{2}}\left(G^{+}(z)+G^{-}(z)\right)=\sum_{q\in\Z}G^{(2)}_{q+\frac{1}{2}}z^{-q-2}.
		\endaligned
	\end{equation}	
	And for $j=1,2$, we have
	\begin{equation}\label{untwisted rela}
		\aligned
		&[L(z_1),L(z_2)]=z_1^{-1}\delta\left(\frac{z_2}{z_1}\right)L^{\prime}(z_2)+2\left(\frac{\partial}{\partial z_2}z_1^{-1}\delta\left(\frac{z_2}{z_1}\right)\right)L(z_2)+\frac{1}{12}\left(\frac{\partial}{\partial z_2}\right)^3\left(z_1^{-1}\delta\left(\frac{z_2}{z_1}\right)\right)C,
		\\&
		[L(z_1),J(z_2)]=z_1^{-1}\delta\left(\frac{z_2}{z_1}\right)J^{\prime}(z_2)+\left(\frac{\partial}{\partial z_2}z_1^{-1}\delta\left(\frac{z_2}{z_1}\right)\right)J(z_2), \\&
		[J(z_1),J(z_2)]=\frac{1}{3}\left(\frac{\partial}{\partial z_2}z_1^{-1}\delta\left(\frac{z_2}{z_1}\right)\right)C,\\&
		[L(z_1), G^{(j)}(z_2)]=z_1^{-1}\delta\left(\frac{z_2}{z_1}\right)G^{(j)\prime}(z_2)+\frac{3}{2}\left(\frac{\partial}{\partial z_2}z_1^{-1}\delta\left(\frac{z_2}{z_1}\right)\right)G^{(j)}(z_2), \\&
		[J(z_1), G^{(1)}(z_2)]=- \mathfrak{i}z_1^{-1}\delta\left(\frac{z_2}{z_1}\right)G^{(2)}(z_2),\\&
		[J(z_1), G^{(2)}(z_2)]= \mathfrak{i} z_1^{-1}\delta\left(\frac{z_2}{z_1}\right)G^{(1)}(z_2),\\&
		[G^{(j)}(z_1),G^{(j)}(z_2)]=2z_1^{-1}\delta\left(\frac{z_2}{z_1}\right)L(z_2)+\frac{1}{3}\left(\frac{\partial}{\partial z_2}\right)^2\left(z_1^{-1}\delta\left(\frac{z_2}{z_1}\right)\right) C,\\&
		[G^{(1)}(z_1),G^{(2)}(z_2)]=-\mathfrak{i} z_1^{-1}\delta\left(\frac{z_2}{z_1}\right)J^{\prime}(z_2)-2\mathfrak{i}\left(\frac{\partial}{\partial z_2}z_1^{-1}\delta\left(\frac{z_2}{z_1}\right)\right)J(z_2),
		\endaligned
	\end{equation}
	which are equivalent to the bracket relations (\ref{N=1}).

	\subsection{Twisted $N=2$ superconformal algebra}
	Now we recall the definition of the  \textit{twisted  $N=2$ superconformal algebra} (the twisted sector). It   is an  infinite dimensional  Lie superalgebra
	$$\mathcal{T}=\bigoplus_{m\in\Z}\C L_m \oplus \bigoplus_{r\in\frac{1}{2}+\Z}\C T_{r} \oplus\bigoplus_{p\in\frac{1}{2}\Z}\C G_{p} \oplus\C C,$$
	which  satisfies  the following Lie super-brackets
	\begin{equation}\label{def1.1}
		\aligned
		&[L_m,L_n]= (m-n)L_{m+n}+\frac{m^{3}-m}{12}\delta_{m+n,0}C,
		\\&
		[L_m,T_r]= -rT_{m+r},\ [T_r,T_s]= \frac{r}{3}\delta_{r+s,0}C,\\&
		[L_m,G_p]= (\frac{m}{2}-p)G_{m+p},\ [T_r,G_p]=G_{r+p}, \\&
		[G_p,G_q]= \left\{\begin{array}{llll}
			(-1)^{2p}\big(2L_{p+q}+\frac{4p^{2}-1}{12}\delta_{p+q,0}C\big)&\mbox{if}\
			p+q\in\Z,\\[4pt]
			(-1)^{2p+1}(p-q)T_{p+q}&\mbox{if}\
			p+q\in\frac{1}{2}+\Z,
		\end{array}\right.
		\endaligned
	\end{equation}
	where $m,n\in\Z, r,s\in\frac{1}{2}+\Z, p,q\in\frac{1}{2}\Z$ and $C$ is a central element.

	Some of the interesting features of   $\T$ are presented as follows. The twisted $N=2$ superconformal algebra consists of the Virasoro  generators
	$L_m$, the Heisenberg generators $T_r$  and  the fermionic generators $G_p$.
	By definition, the even space $\T_{\bar 0}=\mathrm{span}\{L_m,T_r,C\mid m\in\Z,r\in\frac{1}{2}+\Z\}$ and the odd space  $\T_{\bar 1}=\mathrm{span}\{G_p\mid p\in\frac{1}{2}\Z\}$.
	Notice that the even part  $\T_{\bar 0}$   is isomorphic to
	mirror Heisenberg-Virasoro algebra (see \cite{LPXZ}).
	The subalgebra of $\T$ spanned by $\{L_m,G_p,C\mid m\in\Z,p\in\epsilon+\Z\}(\epsilon=0,\frac{1}{2})$  is isomorphic to the well-known super-Virasoro algebra. Clearly, $\C C$ is the center of $\T$.
	Let $\T_m=\mathrm{span}\{L_{m},G_{m},\delta_{m,0}C\mid m\in\Z\}$, $\T_{n-\frac{1}{2}}=\mathrm{span}\{G_{n-\frac{1}{2}},T_{n-\frac{1}{2}}\mid n\in\Z\}$.
	Then $[\T_{\frac{m}{2}},\T_{\frac{n}{2}}]\subset\T_{\frac{m+n}{2}}$, and $\T$ and  $U(\T)$ are $\frac{1}{2}\Z$-graded.
	It is easy to see that $\T$ has the following triangular decomposition:
	$$\T=\T_+\oplus\T_0\oplus\T_-,$$ where $\T_+=\mathrm{span}\{L_{m},T_r,G_{p}\mid m\in\Z_+,r\in\frac{1}{2}+\N,p\in\frac{1}{2}\Z_+\}$ and $\T_-=\mathrm{span}\{L_{m},T_r,G_{p}\mid -m\in\Z_+,-r\in\frac{1}{2}+\N,-p\in\frac{1}{2}\Z_+\}$.

	As a technical treatment in this section, set
	 $$G_{p}^{+}=G_p$$ for $p\in \Z$ and  $$G_{p}^{-}=\mathfrak{i}G_p$$ for $p\in \frac{1}{2}+\Z$. Then for $m,n\in\Z$, $r,s\in\frac{1}{2}+\Z$, the Lie brackets of the twisted $N=2$ superconformal algebra $\mathcal{T}$ can be rewritten as
	\begin{equation*}
		\aligned
		&[L_m, L_n]= (m-n)L_{m+n}+\frac{m^{3}-m}{12}\delta_{m+n,0}C,
		\\&
		[L_m,T_r]= -rT_{m+r},\ [T_r,T_s]= \frac{r}{3}\delta_{r+s,0}C,\\&
		[L_m, G_p^{\pm}]= (\frac{m}{2}-p)G_{m+p}^{\pm},\ [T_r,G_p^{\pm}]=\mp \mathfrak{i} G_{r+p}^{\mp}, \\&
		[G_p^{\pm}, G_q^{\pm}]=  2L_{p+q}+\frac{4p^{2}-1}{12}\delta_{p+q,0}C,\ [G_p^{+},G_q^{-}]=  -\mathfrak{i}(p-q)T_{p+q},
		\endaligned
	\end{equation*}
	where $C$ is still the central element.	Set
	\begin{equation*}
		\aligned
		&L(z)=\sum_{m\in \Z}L_mz^{-m-2},\quad T(z)=-\sum_{r \in\Z}T_{r+\frac{1}{2}}z^{-r-\frac{3}{2}},
		\\&
			G^{+}(z)=\sum_{q\in\Z}G^-_{q+\frac{1}{2}}z^{-q-2}, \quad
	 G^{-}(z)=\sum_{p\in\Z}G_p^{+}z^{-p-\frac{3}{2}}.
		\endaligned
	\end{equation*}
	Then the defining relations above are equivalent to the following identities:
	\begin{equation*}
		\aligned
		&[L(z_1),L(z_2)]=z_1^{-1}\delta\left(\frac{z_2}{z_1}\right)L^{\prime}(z_2)+2\left(\frac{\partial}{\partial z_2}z_1^{-1}\delta\left(\frac{z_2}{z_1}\right)\right)L(z_2)+\frac{1}{12}\left(\frac{\partial}{\partial z_2}\right)^3\left(z_1^{-1}\delta\left(\frac{z_2}{z_1}\right)\right)C,
		\\&
		[L(z_1),T(z_2)]=z_1^{-1}\delta\left(\frac{z_2}{z_1}\right)T^{\prime}(z_2)+\left(\frac{\partial}{\partial z_2}z_1^{-1}\delta\left(\frac{z_2}{z_1}\right)\right)T(z_2), \\&
		[T(z_1),T(z_2)]=\frac{1}{3}\frac{\partial}{\partial z_2}\left(z_1^{-1}\delta\left(\frac{z_2}{z_1}\right)\left(\frac{z_2}{z_1}\right)^{\frac{1}{2}}\right)C,\\&
		[L(z_1), G^{\pm}(z_2)]=z_1^{-1}\delta\left(\frac{z_2}{z_1}\right)G^{{\pm}\prime}(z_2)+\frac{3}{2}\left(\frac{\partial}{\partial z_2}z_1^{-1}\delta\left(\frac{z_2}{z_1}\right)\right)G^{\pm}(z_2), \\&
		[T(z_1), G^{\pm}(z_2)]=\mp \mathfrak{i} z_1^{-1}\delta\left(\frac{z_2}{z_1}\right)\left(\frac{z_2}{z_1}\right)^{\frac{1}{2}} G^{\mp}(z_2),\\&
		[G^+(z_1),G^+(z_2)]=2z_1^{-1}\delta\left(\frac{z_2}{z_1}\right)L(z_2)+\frac{1}{3}\left(\frac{\partial}{\partial z_2}\right)^2\left(z_1^{-1}\delta\left(\frac{z_2}{z_1}\right)\right) C,\\&			[G^-(z_1),G^-(z_2)]=2z_1^{-1}\delta\left(\frac{z_2}{z_1}\right)\left(\frac{z_2}{z_1}\right)^{\frac{1}{2}}L(z_2)+\frac{1}{3}\left(\frac{\partial}{\partial z_2}\right)^2\left(z_1^{-1}\delta\left(\frac{z_2}{z_1}\right)\left(\frac{z_2}{z_1}\right)^{\frac{1}{2}}\right) C,\\&
		[G^+(z_1),G^-(z_2)]=-\mathfrak{i} z_1^{-1}\delta\left(\frac{z_2}{z_1}\right)T^{\prime}(z_2)-2\mathfrak{i}\left(\frac{\partial}{\partial z_2}z_1^{-1}\delta\left(\frac{z_2}{z_1}\right)\right)T(z_2).
		\endaligned
	\end{equation*}
	
	\begin{defi}\rm
		A  $\T$-module is called {\em  restricted} if for any $v\in M$ there exists  $k\in\N$ such that $L_{m}v=T_{r}v=G_pv=0$ for
		$m,p,r> k$. A restricted module is also called a {\em smooth module}.
	\end{defi}
	
	Let $M$ be any restricted module for the twisted $N=2$ superconformal algebra $\mathcal{T}$ with central charge $c$. From the relations above, we see that $\{L(z), T(z), G^+(z),G^-(z)\}$ is a set of mutually local $\Z_2$-twisted vertex operators on $M$. Following the twisted local system theory (c.f. \cite{L2}), let $V$ be the vertex superalgebra generated by $\{L(z), T(z), G(z)^+,G(z)^-\}\cup\{I(z)\}$, where $I(z)$ is the identity operator. Then $(M, Y_M(\cdot,z))$ is a faithful $\psi$-twisted  $V$-module with the generating relations above. By Lemma \ref{2.4}, $Y_V(L(z),z_1)$, $Y_V(T(z),z_1)$, $Y_V(G^+(z),z_1)$, $Y_V(G^-(z),z_1)$ satisfy the untwisted generating relation (\ref{untwisted rela}). Therefore, $V$ is a lowest weight module for the untwisted $N=2$ superconformal algebra $\mathcal{A}$ under the linear map: $$L_m\mapsto L(z)_{m+1},\ J_r\mapsto T(z)_{r},\ G^{(1)}_{p+\frac{1}{2}}\mapsto G^+(z)_{p+1},\ G^{(2)}_{q+\frac{1}{2}}\mapsto G^-(z)_{q+1},\ C\mapsto c$$
	(Note the substitution (\ref{subs}),  the fact that $\left\langle G^{+}_{-\frac{1}{2}} \mathbf{1}\right\rangle+\left\langle G^{-}_{-\frac{1}{2}} \mathbf{1}\right\rangle=\left\langle G^{(1)}_{-\frac{1}{2}} \mathbf{1}\right\rangle=\left\langle G^{(2)}_{-\frac{1}{2}} \mathbf{1}\right\rangle$ in $\bar{M}(c,0,0)$ and the structure of vertex superalgebra $V$).	Thus, $V$ is a vertex operator superalgebra, which is a quotient algebra of $\bar{M}(c,0,0)$. Consequently, $M$ is a weak $\psi$-twisted module for $\bar{M}(c,0,0)$. In fact, we have already proved the desired result.
	\begin{prop}\label{VOA-version}
		Any restricted module $M$ for the  twisted $N=2$ superconformal algebra $\mathcal{T}$ of central charge $c\in\C$ is a weak $\psi$-twisted $\bar{M}(c,0,0)$-module,  where 
		$$Y_M(L_{-2}{\bf 1},z)=L(z),\ Y_M(J_{-1}{\bf 1},z)=T(z),\ Y_M(G^{(1)}_{-\frac{3}{2}}{\bf 1},z)=G^{+}(z),\ Y_M(G^{(2)}_{-\frac{3}{2}}{\bf 1},z)=G^{-}(z),$$
		and vice versa.
	\end{prop}

	\section{Construction of   simple  $\T$-modules}\label{sec4}
	In this section, we present a construction for  simple  restricted modules  over the  twisted  $N=2$  superconformal algebra. From now on, we use the notations and relations (\ref{def1.1}) of $\mathcal{T}$.
	
	We adopt some notations and ideas initially from \cite{ALZ} and \cite{MZ}. Denote  $\mathbf{H}$  the set of all infinite vectors of the form $\mathbf{i}:=(\ldots, i_2, i_1)$ with entries in $\N$,
	satisfying the condition that the number of nonzero entries is finite.
	For
	$k\in\Z_+$, write  $\epsilon_k=(\ldots,\delta_{k,3},\delta_{k,2},\delta_{k,1})$   and $\mathbf{0}=(\ldots, 0, 0)$. Denote
	$$\mathbf{w}(x)=p \quad  \mathrm{if}\quad  0\neq x\in U(\T)_{-p},\quad \forall p\in\frac{1}{2}\Z.$$  For    $\mathbf{i}\in \mathbf{H}$,   set
	\begin{eqnarray}\label{w12112}
		w_{\mathbf{i}}=\cdots \big(G_{-\frac{n-1}{2}}^{i_{2n}}T_{{-n+\frac{1}{2}}}^{i_{2n-1}}\big) \cdots \big(G_{-1}^{i_{6}}T_{{-\frac{5}{2}}}^{i_{5}}\big)\big(G_{-\frac{1}{2}}^{i_4}T_{{-\frac{3}{2}}}^{i_{3}}\big)
		\big(G_0^{i_2}T_{{-\frac{1}{2}}}^{i_1}\big)\in U(\T_-\oplus\T_0),
	\end{eqnarray}
	where $i_{2n},i_{2n-1}\in\N$, $n\in\Z_+$.
	Then $$\mathbf{w}(\mathbf{i}):=\mathbf{w}(w_{\mathbf{i}})
	=\sum_{n=1}^{+\infty}(n-\frac{1}{2})i_{2n-1}+\sum_{n=1}^{+\infty}\frac{n-1}{2}i_{2n}.$$
	Let
	\begin{eqnarray*}
		&&\mathbf{d}(\mathbf{i}):=\mathbf{d}(w_{\mathbf{i}})
		=\sum_{n=1}^{+\infty}(i_{2n-1}+i_{2n}).
	\end{eqnarray*}

	Now we  define the total order   on the set $\mathbf{H}$ (see \cite{MZ}).
	\begin{defi}\rm\label{defin421}
		Denote by $>$ the   {\em reverse  lexicographical  total order}  on  $\mathbf{H}$,  defined as follows:
		\begin{itemize}
			\item[{\rm (1)}]
			$\mathbf{0}$ is the minimum element;
			\item[{\rm (2)}]
			for different nonzero $\mathbf{i},\mathbf{j}\in\mathbf{H}$, we have
			$$\mathbf{j} > \mathbf{i} \ \Longleftrightarrow
			\ \mathrm{ there\ exists} \ m\in\Z_+ \ \mathrm{such \ that} \ (j_n=i_n,\ \forall 1\leq n<m) \ \mathrm{and} \ j_m>i_m.$$
		\end{itemize}
	\end{defi}
	\begin{defi}\rm
		The above reverse lexicographical total order   can be extended to   {\em  principal total order}  $\succ$ on $\mathbf{H}$: for different  $\mathbf{i},\mathbf{j}\in \mathbf{H}$, set $\mathbf{i}\succ\mathbf{j}$ if and only if one of the following conditions is satisfied:
		\begin{itemize}
			\item[{\rm (1)}]  $\mathbf{w}(\mathbf{i})> \mathbf{w}(\mathbf{j})$;

			\item[{\rm (2)}] $\mathbf{w}(\mathbf{i})= \mathbf{w}(\mathbf{j})$ and $\mathbf{d}(\mathbf{i})> \mathbf{d}(\mathbf{j})$;
			
			\item[{\rm (3)}] $\mathbf{w}(\mathbf{i})= \mathbf{w}(\mathbf{j})$, $\mathbf{d}(\mathbf{i})= \mathbf{d}(\mathbf{j})$
			and $\mathbf{i}>\mathbf{j}$.
		\end{itemize}
		
	\end{defi}
	If a simple module   over  $\T$ or one of its  subalgebra  containing the central element $C$,
	we always assume that the action  of $C$  is a scalar  $c$.
	Denote $\b:=\T_+$. Let $M$ be a simple module for $\b$. Then we have the following induced   $\T$-module
	$$\mathrm{Ind}_{\b,c}(M)=U(\T)\otimes_{U(\b)}M.$$
	Using the $\mathrm{PBW}$ Theorem (see \cite{CW}), and $L_{-m}=(-1)^{m}(G_{-\frac{m}{2}})^2$  for $m\in\N$, every element of $\mathrm{Ind}_{\b,c}(M)$ can be uniquely written as
	the following form
	\begin{equation}\label{def2.1}
		\sum_{\mathbf{i}\in\mathbf{H}}v_{\mathbf{i}} \kappa_{\mathbf{i},c},
	\end{equation}
	where $v_{\mathbf{i}}$ is defined as \eqref{w12112},    $\kappa_{\mathbf{i},c}\in M$ and only finitely many of them are nonzero. For any $w\in\mathrm{Ind}_{\b,c}(M)$ as in  \eqref{def2.1}, we denote by $\mathrm{supp}(w)$ the set of all $\mathbf{i}\in \mathbf{H}$  such that $\kappa_{\mathbf{i},c}\neq0$.
	For a nonzero $w\in \mathrm{Ind}_{\b,c}(M)$,
	we write $\mathrm{deg}(w)$  the maximal element in $\mathrm{supp}(w)$ by the principal total order on $\mathbf{H}$,
	which is called the {\em degree} of $w$. Note that $\mathrm{deg}(w)$ is defined only for  $w\neq0$.
	Let
	$\mathbf{w}(w)=\mathbf{w}(\mathrm{deg}(w))$ and $\mathbf{w}(0)=-\infty$.
	For any $p\in\frac{1}{2}\N$, set
	$$\mathrm{supp}_p(w)=\{\mathbf{i}\in \mathrm{supp}(w)\mid \mathbf{w}(\mathbf{i})=p\}.$$

	Now, let us give a  characterization for a simple $\b$-module.
	\begin{lemm}\label{lemm31}
		Let $m\in\Z$, $r,t\in\frac{1}{2}+\mathbb{Z}_+$,    $c\in\C$, $p,p^\prime\in\frac{1}{2}\Z_+$  and $M$ be a  $\b$-module.
		\begin{itemize}
			\item[{\rm (1)}] Assume that  $M$ is a simple module. If $L_{m}M=T_{r}M=0$ for all $m\geq t+\frac{1}{2}$, $r\geq t+1$,
			we obtain
			$G_{p}M=0$ for all $p\geq t+\frac{1}{2}$.
			\item[{\rm (2)}] If  $G_{p}M=0$,
			we get
			$L_{m}M=T_rM=G_{p^\prime}M=0$ for all $m,r\geq p+\frac{1}{2},p^\prime\geq p$.
		\end{itemize}
	\end{lemm}
	\begin{proof}
		{\rm (1)}
		Take any $p\in\frac{1}{2}\N$ with $p\geq t+\frac{1}{2}$.
		From $G_{p}^2M=(-1)^{2p}L_{2p}M=0$, we check that  $Q=G_{p}M$ is a subspace
		of $M$ and $Q\neq M$. We   claim $G_{p+p^\prime}M=0$ for $p^\prime\in\frac{1}{2}\Z_+$. It is clear that $p+p^\prime>t+\frac{1}{2}$. If $p+p^\prime\in\Z_+$, we obtain $G_{p+p^\prime}M=(T_{p+p^\prime-\frac{1}{2}}G_{\frac{1}{2}}
		-G_{\frac{1}{2}}T_{p+p^\prime-\frac{1}{2}})M=0$.
		If $p+p^\prime\in\frac{3}{2}+\Z_+$, one can see that  $G_{p+p^\prime}M=\frac{4}{2p+2p^\prime-3}(L _{p+p^\prime-\frac{1}{2}}G_{\frac{1}{2}}M
		-G_{\frac{1}{2}}L_{p+p^\prime-\frac{1}{2}}M)=0$.  Then for $m\in\N, s\in\frac{1}{2}+\N, q\in\frac{1}{2}\N$, we get
		\begin{eqnarray*}
			&&   L_{m}Q=L_{m}G_{p}M=(G_{p}L_{m}+(\frac{m}{2}-p)G_{m+p})M\subset Q,
			\\&&T_{s}Q=(G_{p}T_{s}+G_{p+s})M\subset Q,
			\\&&
			G_{q}Q=G_{q}G_{p}M=\left\{\begin{array}{llll}
				(-G_{p}G_{q}+(-1)^{2q}2L_{p+q})M  \subset Q&\mbox{if}\
				p+q\in\N,\\[4pt]
				(-G_{p}G_{q}+(-1)^{2q+1}(q-p)T_{p+q})M \subset Q&\mbox{if}\
				p+q\in\frac{1}{2}+\N.
			\end{array}\right.
		\end{eqnarray*}
		This implies that $Q$ is a proper submodule of $M$. By the simplicity of $M$, we have $Q=G_{p}M=0$ for all $p\geq t+\frac{1}{2}$.
		
		{\rm (2)} Take a $p^\prime\in\frac{1}{2}\mathbb{Z}_+$ with $p^\prime\geq p$.  It is clear that $G_{p+\frac{1}{2}}M=[T_{\frac{1}{2}},G_p]M=0$.
		Using recursive method, we check that $G_{p^\prime}M=0$ for all $p^\prime\geq p$.
		Then by \eqref{def1.1}, we have  \begin{equation*}
			\aligned
			&
			0= [G_\frac{1}{2},G_{p^\prime}]M= \left\{\begin{array}{llll}
				-2L_{\frac{1}{2}+p^\prime}M&\mbox{if}\
				\frac{1}{2}+p^\prime\in\Z,\\[4pt]
				(\frac{1}{2}-p^\prime)T_{\frac{1}{2}+p^\prime}M&\mbox{if}\
				p^\prime\in\Z.
			\end{array}\right.
			\endaligned
		\end{equation*}
		The lemma holds.
	\end{proof}

	In the rest of this section, we always set that there exists  $u\in\frac{1}{2}+\Z_+$  such that the $\b$-module $M$   satisfies   the following two conditions:
	\begin{itemize}
		\item[{\rm (i)}] the action of $T_{u}$ on   $M$  is  injective;
		\item[{\rm (ii)}]  $G_{u}M=0$.
	\end{itemize}
	\begin{rema}\label{remar23111}
		From Lemma \ref{lemm31} {\rm (2)} and {\rm (ii)}, one gets
		\begin{eqnarray*}
			L_{m}M=T_{r}M=G_pM=0 \quad \mathrm{for\ all}\ m\geq u+\frac{1}{2}, r\geq u+1,p\geq u.
		\end{eqnarray*}
	\end{rema}
	\begin{lemm}\label{lemm3231}
		Let $m\in\Z, r\in\frac{1}{2}+\Z_+, \varrho\in\frac{1}{2}\Z, c\in\C$. Assume that  $M$ is a  simple $\b$-module and there exists $u\in\frac{1}{2}+\Z_+$ such that $M$ satisfies  the conditions  {\rm (i)} and {\rm (ii)}. For $\varrho\geq u$, let
		$X_\varrho=\left\{\begin{array}{llll}
			L_\varrho \ \mathrm{or}\ G_\varrho&\mbox{if}\
			\varrho\in\Z_+,\\[4pt]
			G_\varrho&\mbox{if}\
			\varrho\in\frac{1}{2}+\Z_+.
		\end{array}\right.$
		Then for any $w\in\mathrm{Ind}_{\b,c}(M)$ with $\mathbf{w}(w)=q\in \frac{1}{2}\Z_+$, we get
		\begin{itemize}
			\item[{\rm (1)}]  $\mathrm{supp}_{q-\varrho+u}(X_\varrho w)\subset\big\{\mathbf{i}-\mathbf{j}\mid \mathbf{i}\in\mathrm{supp}_{q-x}(w),\mathbf{w}(\mathbf{j})=\varrho-u-x,0\leq x\leq \varrho-u\big\}$;
			\item[{\rm (2)}] $\mathbf{w}(X_\varrho w)\leq q-\varrho+u$.
		\end{itemize}
	\end{lemm}
	\begin{proof}
		{\rm (1)}
		We may suppose that $w=v_{\mathbf{i}}\kappa_{\mathbf{i},c}$ with $\mathbf{w}(\mathbf{i})=q$. For   $\varrho \geq u$ and any fixed $X_\varrho$,  by using the relations in \eqref{def1.1}, we can shift the only positive degree term in $[X_\varrho, v_{\mathbf{i}}]$ to the right side, namely,
		$[X_\varrho,v_{\mathbf{i}}]\in \sum_{i\in\{2(q-\varrho),\ldots,2q\}}U(\T_{-})_{-\frac{i}{2}}\T_{-q+\varrho+\frac{i}{2}}$. So
		\begin{eqnarray}\label{113.1222}
			X_\varrho w=[X_\varrho,v_{\mathbf{i}}]\kappa_{\mathbf{i},c}=v_{\varrho-q}\kappa_{\mathbf{i},c}+
			\sum_{\mathbf{j}\in
				\{\mathbf{k}\mid \mathbf{i}-\mathbf{k}\in \mathbf{H},0\leq\mathbf{w}(\mathbf{k})\leq\varrho-\frac{1}{2}\}}v_{\mathbf{i}-\mathbf{j}}\kappa_{\mathbf{j},c}
		\end{eqnarray}
		for some $v_{\varrho-q}\in U(\T)_{\varrho-q}$.
		
		{\rm (2)} Based on  {\rm (1)}, we deduce  $\mathbf{w}(X_\varrho w)\leq\mathbf{w}(w)-\varrho+u$. The lemma holds.
	\end{proof}
	\begin{lemm}\label{lemm33}
		Let $c\in\C$ and $\mathbf{i}\in\mathbf{H}$
		with $\widehat{n}=\mathrm{min}\{k\mid i_k\neq0\}>0$.  Let $M$ be a simple $\b$-module  and  there exists $u\in\frac{1}{2}+\Z_+$ such that $M$
		satisfies the conditions  {\rm (i)} and {\rm (ii)}.
		Thus,
		\begin{itemize}
			\item[{\rm (1)}] if $\widehat{n}=2n-1$ for some $n\in\Z_+$, we have
			\begin{itemize}
				\item[{\rm (a)}]  $\mathrm{deg}(L_{n-\frac{1}{2}+u}v_{\mathbf{i}}\kappa_{\mathbf{i},c})=\mathbf{i}-\epsilon_{\widehat{n}}$;
				\item[{\rm (b)}]
				$\mathbf{i}-\epsilon_{\widehat{n}}\notin\mathrm{supp}
				\big(L_{n-\frac{1}{2}+u}v_{\widetilde{\mathbf{i}}}\kappa_{\widetilde{\mathbf{i}},c}\big)$ for all $\mathbf{i}\succ\widetilde{\mathbf{i}}$.
			\end{itemize}
			\item[{\rm (2)}]
			if $\widehat{n}=2n$ for some $n\in\Z_+$, we have
			\begin{itemize}
				\item[{\rm (a)}]  $\mathrm{deg}(G_{u+\frac{n-1}{2}}v_{\mathbf{i}}\kappa_{\mathbf{i},c})=\mathbf{i}-\epsilon_{\widehat{n}}$;
				\item[{\rm (b)}]
				$\mathbf{i}-\epsilon_{\widehat{n}}\notin\mathrm{supp}
				\big(G_{u+\frac{n-1}{2}}v_{\widetilde{\mathbf{i}}}\kappa_{\widetilde{\mathbf{i}},c}\big)$ for all $\mathbf{i}\succ\widetilde{\mathbf{i}}$.
			\end{itemize}
		\end{itemize}
	\end{lemm}
	\begin{proof}
		{\rm (1)} {\rm (a)}
		Write  $L_{u+n-\frac{1}{2}}v_{\mathbf{i}}\kappa_{\mathbf{i},c}$ as the form of \eqref{113.1222}. Since the action of $T_u$ on $M$ is injective, we see that  the only way to obtain
		the term $v_{\mathbf{i}-\epsilon_{\widehat{n}}}\kappa_{\mathbf{i},c}$ is to commute $L_{u+n-\frac{1}{2}}$ with    $T_{-n+\frac{1}{2}}$, which implies $\mathbf{i}-\epsilon_{\widehat{n}}\in\mathrm{supp}
		\big(L_{n-\frac{1}{2}+u}v_{\mathbf{i}}\kappa_{\mathbf{i},c}\big)$. Note that
		$[L_{u+n-\frac{1}{2}},G_{-n+\frac{1}{2}}]\kappa_{\mathbf{i},c}=0.$
		If there exists      $G_{-n}$ in $v_{\mathbf{i}}$ and $G_{u-\frac{1}{2}}\kappa_{\mathbf{i},c}\neq0$,
		then we get
		the term $v_{\mathbf{i}-\epsilon_{4n+2}}\kappa_{\mathbf{i},c}$   by commuting $L_{u+n-\frac{1}{2}}$ with   $G_{-n}$. This leads to
		$$\mathbf{w}(v_{\mathbf{i}-\epsilon_{4n+2}}\kappa_{\mathbf{i},c})=\mathbf{w}(\mathbf{i})-n
		<\mathbf{w}(\mathbf{i})-n+\frac{1}{2}=\mathbf{w}(\mathbf{i}-\epsilon_{\widehat{n}}).$$
		Combining these with Lemma \ref{lemm3231}, we conclude  that $\mathrm{deg}(L_{n-\frac{1}{2}+u}v_{\mathbf{i}}\kappa_{\mathbf{i},c})=\mathbf{i}-\epsilon_{\widehat{n}}$.
		
		{\rm (b)}
		The proof is divided into the following three cases.
		
		Assume that   $\mathbf{w}(\widetilde{\mathbf{i}}) < \mathbf{w}(\mathbf{i})$. It follows from  Lemma \ref{lemm3231} that  we obtain
		$$\mathbf{w}(L_{n-\frac{1}{2}+u}v_{\widetilde{\mathbf{i}}}\kappa_{\widehat{\mathbf{i}},c})
		\leq \mathbf{w}(\widetilde{\mathbf{i}})-n+\frac{1}{2}<\mathbf{w}(\mathbf{i}-\epsilon_{\widehat{n}})=\mathbf{w}(\mathbf{i})-n+\frac{1}{2}.$$
		Hence, {\rm (b)} holds in this case.
		
		Consider $\mathbf{w}(\widetilde{\mathbf{i}})= \mathbf{w}(\mathbf{i})=p\in\frac{1}{2}\Z$ and $\mathbf{d}(\widetilde{\mathbf{i}})< \mathbf{d}(\mathbf{i})$.
		If the element
		$\mathbf{j}\in\mathrm{supp}
		\big(L_{n-\frac{1}{2}+u}v_{\widetilde{\mathbf{i}}}\kappa_{\widetilde{\mathbf{i}},c}\big)$ is satisfying that $\mathbf{d}(\mathbf{j})< \mathbf{d}(\widetilde{\mathbf{i}})$, then
		$$\mathbf{d}(\mathbf{j})< \mathbf{d}(\widetilde{\mathbf{i}})\leq \mathbf{d}(\mathbf{i})-1=\mathbf{d}(\mathbf{i}-\epsilon_{\widehat{n}}),$$
		which implies that $\mathbf{j}\neq  \mathbf{i}-\epsilon_{\widehat{n}}$.
		If $\mathbf{j}\in\mathrm{supp}
		\big(L_{n-\frac{1}{2}+u}v_{\widetilde{\mathbf{i}}}\kappa_{\widetilde{\mathbf{i}},c}\big)$  is satisfying  that $\mathbf{d}(\mathbf{j})=\mathbf{d}(\widetilde{\mathbf{i}})$, then
		such $\mathbf{j}$ can only be obtained by commuting $L_{n-\frac{1}{2}+u}$ with some $T_{-r},r\in\frac{1}{2}+\Z_+$, where $r>n-\frac{1}{2}+u$.
		Then it is easy to see that
		$$\mathbf{w}(\mathbf{j})=\mathbf{w}(\widetilde{\mathbf{i}})-n+\frac{1}{2}-u < \mathbf{w}\widetilde{(\mathbf{i}})-n+\frac{1}{2}= \mathbf{w}(\mathbf{i})-n+\frac{1}{2}=\mathbf{w}(\mathbf{i}-\epsilon_{\widehat{n}}),$$
		which   shows   $\mathbf{j}\neq \mathbf{i}-\epsilon_{\widehat{n}}$. Thus
		{\rm (b)} also holds in this case.
		Let $\widetilde{n}=\mathrm{min}\{n\mid \widetilde{i}_n\neq0\}$ be in $\widetilde{\mathbf{i}}$.
		If $\widetilde{n}=\widehat{n}$, then by {\rm (a)}, $\mathrm{deg}(L_{n-\frac{1}{2}+u}v_{\widetilde{\mathbf{i}}}\kappa_{\widetilde{\mathbf{i}},c})
		=\widetilde{\mathbf{i}}-\epsilon_{\widehat{n}}$, we also get {\rm (b)} in this case.

		Consider the last case $\mathbf{w}(\widetilde{\mathbf{i}})=\mathbf{w}(\mathbf{i})= p\in\frac{1}{2}\Z$,  $\mathbf{d}(\widetilde{\mathbf{i}})=\mathbf{d}(\mathbf{i})$ and $\widetilde{n}>\widehat{n}$.
		Then by Lemma \ref{lemm3231}, it is easy to check  that $\mathbf{w}(L_{n-\frac{1}{2}+u}v_{\widetilde{\mathbf{i}}}\kappa_{\widetilde{\mathbf{i}},c})<p-n+\frac{1}{2}
		=\mathbf{w}(\mathbf{i}-\epsilon_{\widehat{n}})$.
		This completes the proof of (b).
		
		{\rm (2)} The proof of {\rm (2)} is similar to {\rm (1)}. For readability, we still present  the
		complete proof processes.
		
		{\rm (a)} We write  $G_{u+\frac{n-1}{2}}v_{\mathbf{i}}\kappa_{\mathbf{i},c}$ as  \eqref{113.1222}. Since the action of $T_u$ on $M$ is injective, we note  that  the only way to give
		the term $v_{\mathbf{i}-\epsilon_{\widehat{n}}}\kappa_{\mathbf{i},c}$ is to commute $G_{u+\frac{n-1}{2}}$ with   $G_{-\frac{n-1}{2}}$, which shows $\mathbf{i}-\epsilon_{\widehat{n}}\in\mathrm{supp}
		\big(G_{u+\frac{n-1}{2}}v_{\mathbf{i}}\kappa_{\mathbf{i},c}\big)$.
		Combining this with Lemma \ref{lemm3231}, we deduce   $\mathrm{deg}(G_{u+\frac{n-1}{2}}v_{\mathbf{i}}\kappa_{\mathbf{i},c})=\mathbf{i}-\epsilon_{\widehat{n}}$.
		
		{\rm (b)}
		We first consider   $\mathbf{w}(\widetilde{\mathbf{i}}) < \mathbf{w}(\mathbf{i})$. According to   Lemma \ref{lemm3231},   we have
		$$\mathbf{w}(G_{u+\frac{n-1}{2}}v_{\widetilde{\mathbf{i}}}\kappa_{\widetilde{\mathbf{i}},c})
		\leq \mathbf{w}(\widetilde{\mathbf{i}})-\frac{n-1}{2}<\mathbf{w}(\mathbf{i}-\epsilon_{\widehat{n}})
		=\mathbf{w}(\mathbf{i})-\frac{n-1}{2}.$$
		So, {\rm (b)} holds in this case.
		
		Now consider $\mathbf{w}(\widetilde{\mathbf{i}})= \mathbf{w}(\mathbf{i})=p\in\frac{1}{2}\Z$ and $\mathbf{d}(\widetilde{\mathbf{i}})< \mathbf{d}(\mathbf{i})$.
		If  there exists an element
		$\mathbf{j}\in\mathrm{supp}
		\big(G_{u+\frac{n-1}{2}}w_{\widetilde{\mathbf{i}}}\kappa_{\widetilde{\mathbf{i}},c}\big)$  such that $\mathbf{d}(\mathbf{j})< \mathbf{d}(\widetilde{\mathbf{i}})$, then
		$$\mathbf{d}(\mathbf{j})< \mathbf{d}(\widetilde{\mathbf{i}})\leq \mathbf{d}(\mathbf{i})-1=\mathbf{d}(\mathbf{i}-\epsilon_{\widehat{n}}),$$
		which implies $\mathbf{j}\neq  \mathbf{i}-\epsilon_{\widehat{n}}$.
		If there exists $\mathbf{j}\in\mathrm{supp}
		\big(G_{u+\frac{n-1}{2}}w_{\widetilde{\mathbf{i}}}\kappa_{\widetilde{\mathbf{i}},c}\big)$    such that $\mathbf{d}(\mathbf{j})=\mathbf{d}(\widetilde{\mathbf{i}})$, then
		such $\mathbf{j}$ can only be given by commuting $G_{u+\frac{n-1}{2}}$ with some $G_{-p},p\in\frac{1}{2}\N$, where $p>\frac{n-1}{2}+u$.
		Thus we get
		$$\mathbf{w}(\mathbf{j})=\mathbf{w}(\widetilde{\mathbf{i}})-\frac{n-1}{2}-u < \mathbf{w}\widetilde{(\mathbf{i}})-\frac{n-1}{2}= \mathbf{w}(\mathbf{i})-\frac{n-1}{2}=\mathbf{w}(\mathbf{i}-\epsilon_{\widehat{n}}),$$
		which   shows $\mathbf{j}\neq \mathbf{i}-\epsilon_{\widehat{n}}$. So,
		{\rm (b)} also holds in this case.
		Denote $\widetilde{n}=\mathrm{min}\{n\mid \widetilde{i}_n\neq0\}$.
		If $\widetilde{n}=\widehat{n}$, then by {\rm (a)}, $\mathrm{deg}(G_{u+\frac{n-1}{2}}v_{\widetilde{\mathbf{i}}}\kappa_{\widetilde{\mathbf{i}},c})
		=\widetilde{\mathbf{i}}-\epsilon_{\widehat{n}}$, we also have {\rm (b)} in this case.

		At last, we   consider   $\mathbf{w}(\widetilde{\mathbf{i}})=\mathbf{w}(\mathbf{i})= p\in\frac{1}{2}\Z$,  $\mathbf{d}(\widetilde{\mathbf{i}})=\mathbf{d}(\mathbf{i})$ and $\widetilde{n}>\widehat{n}$.
		Then based on  Lemma \ref{lemm3231}, we can deduce that $\mathbf{w}(G_{u+\frac{n-1}{2}}v_{\widetilde{\mathbf{i}}}\kappa_{\widetilde{\mathbf{i}},c})<p-\frac{n-1}{2}
		=\mathbf{w}(\mathbf{i}-\epsilon_{\widehat{n}})$.
		We complete   the proof of (b).
	\end{proof}
	
	As an immediate consequence of Lemma  \ref{lemm33}  we  have:
	\begin{coro}\label{3344}
		Let $ c\in\C$. Assume that $M$ is a simple $\b$-module  and there exists $u\in\frac{1}{2}+\Z_+$ such that $M$ satisfies the conditions  {\rm (i)} and {\rm (ii)}.
		Let $0\neq w\in\mathrm{Ind}_{\b,c}(M)$ and $\mathrm{deg}(w)=w_{\mathbf{i}}$ for $\mathbf{i}\in\mathbf{H}$. Write  $\widehat{n}=\mathrm{min}\{n\mid i_n\neq0\}>0$.
		\begin{itemize}
			\item[{\rm (1)}]  If $\widehat{n}=2n-1$ for some $n\in\Z_+$, then   $L_{u+n-\frac{1}{2}}w\neq0$;
			\item[{\rm (2)}]
			If $\widehat{n}=2n$ for some $n\in\Z_+$, then $G_{u+\frac{n-1}{2}}w\neq0$.
		\end{itemize}
	\end{coro}
	Now  we give the main statement of this section.
	\begin{theo}\label{th1}
		Let  $c\in\C$ and $M$ be a simple $\b$-module. Assume that  there exists $u\in\frac{1}{2}+\Z_+$ such that $M$ satisfies the conditions  {\rm (i)} and {\rm (ii)}.   Then we obtain that
		$\mathrm{Ind}_{\b,c}(M)$ is a simple $\T$-module.
	\end{theo}
	\begin{proof}
		According to  Corollary \ref{3344},  for any  $0\neq v\in\mathrm{Ind}_{\b,c}(M)$   we get   a nonzero element in
		$U(\T)v\cap M$, which implies the simplicity of $\mathrm{Ind}_{\b,c}(M)$.
		The theorem follows.
	\end{proof}

	\section{Characterization of simple restricted $\T$-modules}
	In this section,  a characterization under certain conditions for simple restricted  modules over the twisted $N=2$ superconformal  algebra  is presented.
	For $m\in\Z, r,u\in\frac{1}{2}+\Z_+, p\in\frac{1}{2}\Z$,
	we denote   $$\T^{(u)}=\bigoplus_{p\geq u}\C G_{p}\oplus\bigoplus_{m\geq u+\frac{1}{2}}\C L_{m}\oplus\bigoplus_{r\geq u+1}\C T_r.$$
	We  show several equivalent descriptions for simple
	restricted $\T$-modules as follows.
	\begin{lemm}\label{th3}
		Let $m\in\Z, r\in\frac{1}{2}+\Z, p\in\frac{1}{2}\Z$. Let  $\V$ be a simple  $\T$-module.  Then the following conditions are equivalent:
		\begin{itemize}
			\item[{\rm (1)}] $\V$ is a restricted $\T$-module.

			\item[{\rm (2)}]   There exists $u\in\frac{1}{2}+\Z_+$ such that the actions of $L_{m},T_r,G_{p}$ for $m\geq u+\frac{1}{2}, r\geq u+1,  p\geq u$ on $\V$ are locally nilpotent.
			
			\item[{\rm (3)}]  There exists $u\in\frac{1}{2}+\Z_+$ such that the actions of $L_{m},T_r,G_{p}$ for $m\geq u+\frac{1}{2}, r\geq u+1,  p\geq u$ on $\V$ are locally finite.

			\item[{\rm (4)}]    There exists $u\in\frac{1}{2}+\Z_+$ such that  $\V$ is a  locally nilpotent $\T^{(u)}$-module.
			
			\item[{\rm (5)}]  There exists $u\in\frac{1}{2}+\Z_+$ such that  $\V$ is a  locally finite $\T^{(u)}$-module.
		\end{itemize}
	\end{lemm}

	\begin{proof}
		
		It is easy to see that $(4)\Rightarrow(5)\Rightarrow(3)$  and $(4)\Rightarrow(2)\Rightarrow(3)$.

		First consider $(1)\Rightarrow(4)$.
		According to the definition of restricted module, for any nonzero element $w\in \V$, there exists
		$u^\prime\in\frac{1}{2}+\Z_+$ such that $L_{m}w=T_rw=G_pw=0$ for $r\geq u^\prime+1, m\geq u^\prime+\frac{1}{2}, p\geq u^\prime$. By the simplicity of $\V$, we obtain  $\V=U(\T)w$. Then it follows from the $\mathrm{PBW}$ Theorem  that $\V$ is a locally
		nilpotent  module over $\T^{(u)}$ for $u\geq u^\prime$.
		
		Now    consider $(3)\Rightarrow(1)$.
		By the definition of locally finite of (3),
		we can take a nonzero element $w\in \V$ such that $L_{u+\frac{1}{2}}w=\lambda w$ for some $\lambda\in \C$.
		Let  $m\in\Z$, $p\in\frac{1}{2}\Z, r\in\frac{1}{2}+\Z$ with $r\geq u+1, m\geq u+\frac{1}{2}, p\geq u$. We  denote the following three finite-dimensional vector spaces:
		\begin{eqnarray*}
			&&\M_L=\sum_{k\in\N}\C L_{u+\frac{1}{2}}^kL_{m}w={U}(\C L_{u+\frac{1}{2}})L_{m}w,
			\\&&\M_T=\sum_{k\in\N}\C L_{u+\frac{1}{2}}^kT_{r}w={U}(\C L_{u+\frac{1}{2}})T_{r}w,
			\\&&\M_G=\sum_{k\in\N}\C L_{u+\frac{1}{2}}^kG_{p}w={U}(\C L_{u+\frac{1}{2}})G_{p}w.
		\end{eqnarray*}
		From the definition of $\T$ and  $k\in\N$, one has
		\begin{eqnarray*}
			&&L_{m+k(u+\frac{1}{2})}w\in \M_L\Rightarrow L_{m+(k+1)(u+\frac{1}{2})}w\in \M_L,\quad \forall \ m\geq u+\frac{1}{2},
			\\&&T_{r+k(u+\frac{1}{2})}w\in \M_T\Rightarrow T_{r+(k+1)(u+\frac{1}{2})}w\in \M_T,\quad \forall \ r\geq u+1,
			\\&&G_{p+k(u+1)}w\in \M_G\Rightarrow G_{p+(k+1)(u+\frac{1}{2})}w\in \M_G,\quad \forall \ p\geq u.
		\end{eqnarray*}
		Using induction     on  $k\in\N$, we have
		$L_{m+k(u+\frac{1}{2})}w\in \M_L$, $T_{r+k(u+\frac{1}{2})}w\in \M_T$ and $G_{p+k(u+\frac{1}{2})}w\in \M_G$.
		In fact,  $\sum_{k\in\N}\C L_{m+k(u+\frac{1}{2})}w$, $\sum_{k\in\N}\C T_{r+k(u+\frac{1}{2})}w$ and $\sum_{k\in\N}\C G_{p+k(u+\frac{1}{2})}w$ are all
		finite-dimensional for    $r\geq u+1, m\geq u+\frac{1}{2}, p\geq u$. Hence,
		\begin{eqnarray*}
			&&\sum_{i\in\N}\C L_{u+\frac{1}{2}+i}w=\C L_{u+\frac{1}{2}}w+\sum_{m=u+\frac{3}{2}}^{2u+1}\big(\sum_{k\in\N}\C L_{m+k(u+\frac{1}{2})}w\big),
			\\&&\sum_{i\in\N}\C T_{u+1+i}w=\C T_{u+1}w+\sum_{r=u+2}^{2u+\frac{3}{2}}\big(\sum_{k\in\N}\C T_{r+k(u+\frac{1}{2})}w\big),
			\\&&\sum_{i\in\N}\C G_{u+i}w=\C G_{u}w+\sum_{p=u+1}^{2u+\frac{1}{2}}\big(\sum_{k\in\N}\C G_{p+k(u+\frac{1}{2})}w\big),
			\\&&\sum_{i\in\N}\C G_{u+\frac{1}{2}+i}w=\C G_{u+\frac{1}{2}}w+\sum_{p=u+\frac{3}{2}}^{2u+1}\big(\sum_{k\in\N}\C G_{p+k(u+\frac{1}{2})}w\big)
		\end{eqnarray*}
		are all finite-dimensional. Now we can choose $t\in\N$ such that
		\begin{eqnarray*}
			&&\sum_{i\in\N}\C L_{u+\frac{1}{2}+i}w=\sum_{i=0}^{t}\C L_{u+\frac{1}{2}+i}w,
			\\&&\sum_{i\in\N}\C T_{u+1+i}w=\sum_{i=0}^{t}\C T_{u+1+i}w,
			\\&&\sum_{i\in\N}\C G_{u+\frac{i}{2}}w=\sum_{i=0}^{t}\C G_{u+\frac{i}{2}}w.
		\end{eqnarray*}
		Denote  $$W^\prime=\sum_{x_0,x_1,\ldots,x_t,y_0,y_1,\ldots,y_t,z_0,z_1,\ldots,z_t\in\N}\C L_{u+\frac{1}{2}}^{x_0}L_{u+\frac{3}{2}}^{x_1}\cdots
		L_{u+\frac{1}{2}+t}^{x_t}T_{u}^{y_0}T_{u+1}^{y_1}\cdots
		T_{u+t}^{y_t} G_{u}^{z_0}G_{u+\frac{1}{2}}^{z_1}\cdots
		G_{u+\frac{t}{2}}^{z_t}  w.$$
		It is easy to check   that
		$W^\prime$ is a (finite-dimensional) $\T^{(u)}$-module.

		It follows that we can  choose
		a minimal $n\in\N$ such that
		\begin{equation}\label{lm3.3}
			(L_m+a_1L_{m+1}+\cdots + a_{n}L_{m+n})W^\prime=0
		\end{equation}
		for some $m\in\Z$ with $m\geq u+\frac{1}{2}$ and  $a_i\in \C,i=1,\ldots,n$.
		Applying $\mathrm{ad}(L_{m})$ to \eqref{lm3.3}, we have
		$$(a_1[L_{m},L_{m+1}]+\cdots +a_{n}[L_{m},L_{m+n}])W^\prime=0,$$
		which shows $n=0$. That is to say, $L_mW^\prime=0$  for some $m\geq u+\frac{1}{2}$.
		Therefore, for any $k\geq m+1,r\geq u+1$, one has
		\begin{eqnarray*}
			&&L_{k+m}W^\prime=\frac{1}{k-m}\big(L_kL_m-L_mL_k\big)W^\prime=0,
			\\&&T_{r+m}W^\prime=\frac{1}{r}\big(T_rL_m-L_mT_r\big)W^\prime=0.
		\end{eqnarray*}
		By using this, it is easy to get
		\begin{eqnarray*}
			&&G_{r+m+\frac{1}{2}}W^\prime=\big(T_{m+r}G_{\frac{1}{2}}-G_{\frac{1}{2}}T_{m+r}\big)W^\prime=0,
			\\&&G_{k+m+\frac{1}{2}}W^\prime=\frac{2}{m+k-1}\big(L_{m+k}G_{\frac{1}{2}}-G_{\frac{1}{2}}L_{m+k}\big)W^\prime=0,
		\end{eqnarray*}
		namely,  $G_{p+m+\frac{1}{2}}W^\prime=$ for   all $p>m$.
		
		Therefore,  there exists a nonzero element  $w^\prime\in\V$ such that $L_kw^\prime=T_rw^\prime=G_pw^\prime=0$  for all
		$k,r,p>2m+1$. Note that $\V=U(\T)w^\prime$. Based on the $\mathrm{PBW}$ theorem,  we check that
		each element of $\T$ can be written as a linear combination  of vectors
		$$\cdots G_{2m+\frac{1}{2}}^{i_{2m+\frac{1}{2}}}G_{2m+1}^{i_{2m+1}}\cdots T_{2m-\frac{1}{2}}^{j_{2m- \frac{1}{2}}}T_{2m+\frac{1}{2}}^{j_{2m+\frac{1}{2}}}\cdots L_{2m}^{l_{2m}}L_{2m+1}^{l_{2m+1}}w^\prime.$$
		Then for any $w \in \V$, there exists sufficiently large $t\in\N$ such that $L_mw=T_rw=G_pw=0$ for any $m,p,r>t$, which shows that $\V$ is a restricted $\T$-module.   We prove the lemma.

	\end{proof}
	The main results of this section can be given as follows.
	\begin{theo}\label{th2}
		Let $c\in\C$ and  $\V$ be a simple restricted  $\T$-module. Assume that there exists $s\in\frac{1}{2}+\Z_+$ such that the action of $T_{s}$ on   $\V$  is  injective.
		\begin{itemize}
			\item[{\rm (1)}] Then  there
			exists the smallest $t\in\frac{1}{2}+\Z_+$ with $t\geq s$ such that
			$$
			\M_{t}=\Big\{w\in \V\mid L_{\hat{t}+\frac{1}{2}}w=T_{\hat{t}+1}w=G_{\hat{t}}w=0, \forall\hat{t}>t-\frac{1}{2}\Big\}\neq0.$$ In particular, $M:=\M_{t}$ is a $\b$-module.
			\item[{\rm (2)}] Then   $\V$ can be  described by $M$ as follows. \begin{itemize}
				\item[{\rm (a)}] If $t=s$, then $M$ is a simple $ \b$-module, and
				$\V\cong\mathrm{Ind}_{\b,c}(M)$.
				
				\item[{\rm (b)}]   If $t>s$, suppose that the action of  $T_t$    on $M$ is injective.  Then $M$ is a simple $ \b$-module, and
				$\V\cong\mathrm{Ind}_{\b,c}(M)$.
			\end{itemize}
		\end{itemize}
	\end{theo}
	\begin{proof}

		(1) We note that there exists $s\in\frac{1}{2}+\Z_+$ such that $T_{s}$  acts injectively on   $\V$.
		For   $t^\prime\in\frac{1}{2}+\Z_+$,  we  define the following   vector space
		\begin{eqnarray*}
			\M_{t^\prime}=\Big\{w\in \V\mid L_{\hat{t}+\frac{1}{2}}w=T_{\hat{t}+1}w=G_{\hat{t}}w=0, \forall \hat{t}>t^\prime-\frac{1}{2}\Big\}.
		\end{eqnarray*}
		By the definition of restricted modules, we see that $\M_{t^\prime}\neq0$ for sufficiently large $t^\prime\geq s$.
		Since    $T_s$ acts injectively on   $\V$, we conclude that $\M_{t^\prime}=0$ for all $t^\prime<s$.
		Thus we can find the smallest   $t\in\frac{1}{2}+\Z_+$ such that $M:= \M_t\neq0$.
		\begin{clai}
			For $s_1,s_2\in\{\frac{1}{2},\ldots,t-\frac{1}{2}\}$ with $s_1+s_2=t$, we obtain $G_{s_1}M\neq0$ and $G_{s_2}M\neq0$.
		\end{clai}
		Suppose $G_{s_1}M=0$ or $G_{s_2}M=0$. For any $t\in\frac{1}{2}+\Z_+$, we have  $$T_tM=\frac{(-1)^{2s_1+1}}{s_1-s_2}(G_{s_1}G_{s_2}+G_{s_2}G_{s_1})M=0,$$ which shows  $M=0$. This contradicts with  $M\neq0$. The claim holds.
		Choosing $q\in\frac{1}{2}\Z_+,w\in M$,  it is easy to see that
		\begin{eqnarray*}
			&&
			G_{\hat{t}}(G_{q}w)=
			\left\{\begin{array}{llll}
				-2L_{\hat{t}+q}w=0&\mbox{if}\
				\hat{t}+q\in\N,\\[4pt]
				(\hat{t}-q)T_{\hat{t}+q}w=0&\mbox{if}\
				\hat{t}+q\in\frac{1}{2}+\N,
			\end{array}\right.
			\\&&T_{\hat{t}+1}(G_{q}w)=G_{\hat{t}+1+q}w=0,
			\\&&L_{\hat{t}+\frac{1}{2}}(G_{q}w)=(\frac{\hat{t}}{2}+\frac{1}{4}-q)G_{\hat{t}+\frac{1}{2}+q}w=0.
		\end{eqnarray*}
		This gives  $G_{q}w\in M$ for all $q\in \frac{1}{2}\Z_+$.  By the similar method, we can deduce
		$
		T_{p}w, L_{p+\frac{1}{2}}w
		$ are
		all in $M$ for all $p\in \frac{1}{2}+\N$. So, $M$ is a  $\b$-module.

		(2)  (a)
		Since $\V$ is simple
		and generated by $M$,  we know that  there exists a canonical $\T$-modules  epimorphism
		$$\Phi:\mathrm{Ind}_{\b,c}(M) \rightarrow \V, \quad \Phi(1\otimes v)=v,\quad \forall  v\in M.$$
		So, we need only to prove that $\Phi$ is   injective, namely, $\Phi$ as the canonical map is bijective.  Let $F=\mathrm{ker}(\Phi)$. Clearly, $F\cap (1\otimes M)=0$. Suppose	$F\neq0$, we can choose $0\neq w\in F\setminus (1\otimes M)$ such that $\mathrm{deg}(w)=\mathbf{i}$ is minimal possible.
		Observe    that $F$ is a $\T$-submodule of $\mathrm{Ind}_{\b,c}(M)$.  By the definition of $M$, it is clear  that the action of  $T_{t}$  on $M$ is injective.
		Using  an identical process in Lemma \ref{lemm33}, we   obtain a new vector $\eta\in F$  with $\mathrm{deg}(\eta)\prec\mathbf{i}$, which  shows a contradiction. This forces $F=0$,
		i.e., $\V\cong \mathrm{Ind}_{\b,c}(M)$. From the property of induced modules, we see that $M$ is   simple.
		
		By the similar method used in (a), we have (b). This completes  the proof.
	\end{proof}

	\section{Examples}
	In this section, we show some examples of new simple restricted (non-weight) modules  for the twisted  $N=2$ superconformal algebra.
	\subsection{Generalized Whittaker modules}\label{ex789901}
	Denote $$\mathfrak{T}=\bigoplus_{m\in\Z_+}L_m\oplus\bigoplus_{r\in\frac{1}{2}+\Z_+}
	T_r\oplus\bigoplus_{p\in\frac{1}{2}(\Z_+\setminus\{1\})}G_p.$$
	Note that $\b=\mathfrak{T}\oplus \C T_{\frac{1}{2}}\oplus \C G_{\frac{1}{2}}$. Let  $\phi:\mathfrak{T}\rightarrow\C$ is a non-trivial Lie superalgebra homomorphism and let $\phi(G_{\frac{3}{2}})=\phi(L_2)=0$.  Then we have  $\phi(L_m)=\phi(T_{r})=\phi(G_p)=0$ for $m\geq2,r\geq\frac{5}{2},p\geq\frac{3}{2}$. Let $\mathfrak{t}_\phi=\C v_{\bar0}\oplus\C v_{\bar1}$ be  a $2$-dimensional vector space with
	$$xv_{\bar0}=\phi(x)v_{\bar0}, \ v_{\bar1}=G_{\frac{1}{2}}v_{\bar0}, \  Cv_{\bar0}=cv_{\bar0},\ Cv_{\bar1}=cv_{\bar1}$$  for all $x\in \mathfrak{T}.$
	Clearly, if $\phi(T_{\frac{3}{2}})\neq0$,   $\mathfrak{t}_{\phi}$ is a simple $\mathfrak{T}$-module and  $\mathrm{dim}(\mathfrak{t}_{\phi})=2$.
	For $m\in\Z, p\in\frac{1}{2}\Z, r\in\frac{1}{2}+\Z$, we denote
	$$\mathfrak{p}=\bigoplus_{m\geq1}\C L_m\oplus\bigoplus_{p>\frac{1}{2}}\C G_p\oplus\bigoplus_{r>0}\C T_r\oplus\C C
	\ \mathrm{and} \   \mathfrak{b}=\p\oplus\C G_\frac{1}{2}.$$
	Now we consider the induced module
	$$W_\phi=U(\p)\otimes_{U(\mathfrak{T})} \mathfrak{t}_\phi=\C[T_{\frac{1}{2}}]v_{\bar0}\oplus \C[T_{\frac{1}{2}}](G_{\frac{1}{2}}v_{\bar0}).$$
	It is easy to check that $W_\phi$ is a simple $\p$-module if $\phi(T_{\frac{3}{2}})\neq0$.
	Let $V_{\phi}=U(\b)\otimes_{U(\p)}W_\phi$.
	Then the simple induced $\T$-modules $\mathrm{Ind}_{\b,c}(V_{\phi})$  in Theorem \ref{th1}  are  generalized Whittaker modules. In fact,   if
	$\phi(T_{\frac{3}{2}})=0$, then $U(\T)(\C[T_{\frac{1}{2}}]v_{\bar0})$ and $U(\T)(\C[T_{\frac{1}{2}}](G_{\frac{1}{2}}v_{\bar0}))$ are nonzero proper submodules of $\mathrm{Ind}_{\b,c}(V_\phi)$.

	\subsection{High order Whittaker  modules}\label{ex78990}
	Let $\phi_s$ be a  non-trivial Lie superalgebra homomorphism $\phi_s:\T^{(s)}\rightarrow\C$ for  $s\in\frac{1}{2}+\Z_+$ and let $\phi_s(L_{2s+1})=0$.
	Then we get   $\phi_s(L_m)=\phi_s(G_p)=\phi_s(T_{r})=0$ for $m\geq2s+1,p\geq2s+\frac{1}{2}, r\geq 2s+\frac{3}{2}$.
	Suppose that $\mathfrak{t}_{\phi_s}=\C v_{\bar0}\oplus\C v_{\bar1}$ is a $2$-dimensional vector space with
	$$xv_{\bar0}=\phi(x)v_{\bar0},\ v_{\bar1}=G_{\frac{1}{2}}v_{\bar0}, \ Cv_{\bar0}=cv_{\bar0},\ Cv_{\bar1}=cv_{\bar1}$$  for all $x\in \T^{(s)}.$
	If $\phi_s(T_{2s+\frac{1}{2}})\neq0$,  $\mathfrak{t}_{\phi_s}$ is a simple $\T^{(s)}$-module and  $\mathrm{dim}(\mathfrak{t}_{\phi_s})=2$.
	Consider the induced module
	$$W_{\phi_s}=U(\mathfrak{p})\otimes_{U(\T^{(s)})} \mathfrak{t}_{\phi_s}.$$
	Then $W_{\phi_s}$ is a simple $\p$-module if $\phi_s(T_{2s+\frac{1}{2}})\neq0$ and  $\mathrm{dim}(W_{\phi_s})=2$. Denote $V_{\phi_s}=U(\b)\otimes_{U(\p)}W_{\phi_s}$.
	The corresponding simple  $\T$-modules $\mathrm{Ind}_{\b,c}(V_{\phi_s})$  in Theorem \ref{th1}  are exactly the high order Whittaker modules.

	\subsection{$\T$-modules from simple induced   $\b\oplus\T_0$-modules}
	Denote $B=\b\oplus\T_0$. Suppose that  $M$ is a simple $\b$-module  and  satisfies   the   conditions (i) and (ii) appeared in Section \ref{sec4}.
	Then we have induced  $B$-modules $\mathrm{Ind}_{\b,c}^{B}(M)$.
	\begin{prop}
		The induced $B$-module $\mathrm{Ind}_{\b,c}^{B}(M)$ is simple.
	\end{prop}
	\begin{proof}
		By the $\mathrm{PBW}$ Theorem and $L_0=G_0^2$, for any $v\in \mathrm{Ind}_{\b,c}^{B}(M)\setminus M$, we can write $v$ in the form
		$$\sum_{i_1\in\N}G^{i_1}_0v_{i_1},$$
		where $v_{i_1}\in M$. Suppose that
		$\mathrm{supp}(v_{i_1})$ is the set of all $i_1$ with $v_{i_1}\neq0$, and $\mathrm{deg}(v_{i_1})$ is the maximal element of $\mathrm{supp}(v_{i_1})$
		with respect to reverse lexicographic order on $\Z$.  Using the similar method in Lemma \ref{lemm33} (2), we can check that $\mathrm{Ind}_{\b,c}^{B}(M)$ is a simple $B$-module.
	\end{proof}
	From the above setting, we obtain a class of simple induced  $\T$-modules $\mathrm{Ind}_{B,c}(\mathrm{Ind}_{\b,c}^{B}(M))$, where  $M$ is a simple $\b$-module  and  satisfies    (i) and (ii) in Section \ref{sec4}.

	\subsection{Whittaker modules}\label{ex7899012233}
	In \cite{ALTY},  the authors studied some simple non $
	\mathbb{Z}_2$-graded Whittaker modules for $N=1$ Ramond algebras, which provides irreducible modules for  orbifolds for $N=1$ Neveu-Schwarz vertex superalgebras.
	In this section, we will consider those similar Whittaker modules over   twisted $N=2$ superconformal algebras.  We first recall the definition of Whittaker modules.
	\begin{defi}\rm
		Let $c\in\C$. Assume that  $\phi: \mathfrak{b}\rightarrow \C$ is a  Lie superalgebra homomorphism with $\phi(G_{\frac{1}{2}})=0$ which will be called a {\em Whittaker function}.
		A $\T$-module $M$ is called a {\em Whittaker module} of type $(\phi,c)$ if
		\begin{itemize}
			\item[{\rm (1)}]
			$M$  is generated by a homogeneous vector $w$;
			\item[{\rm (2)}]
			$xw=\phi(x)w$ for any $x\in\mathfrak{b}$;
			\item[{\rm (3)}]  $Cw = cw$,
		\end{itemize}
		where $w$ is called a {\it Whittaker vector} of $M$.
	\end{defi}
	The following lemma can be obtained  immediately by  \eqref{def1.1}.
	\begin{lemm}\label{lemma3.2}
		Let $m\in\Z,r\in\frac{1}{2}+\Z,p\in\frac{1}{2}\Z$ and $\phi:\mathfrak{b}\rightarrow\C$ be a  Lie superalgebra homomorphism with $\phi(G_{\frac{1}{2}})=0$. Then
		we have $\phi(L_m)=\phi(T_r)=\phi(G_p)=0$
		for $r>\frac{1}{2},m,p>0$.
	\end{lemm}
	For a nontrivial Lie superalgebra homomorphism $\phi:\mathfrak{b}\rightarrow\C$ with $\phi(G_{\frac{1}{2}})=0$, we define $\C v_{\phi,c}$    as the one-dimensional
	$\mathfrak{b}\oplus\C C$-module given by $xv_{\phi,c}=\phi(x)v_{\phi,c}$  and $C v_{\phi,c}=c v_{\phi,c}$ for $x\in\mathfrak{b},c\in\C$. Then we have an induced   $\T$-modules $$W(\phi,c)=\mathrm{Ind}_{\mathfrak{b}\oplus\C C}^{\T}\C v_{\phi,c}=U(\T)\otimes_{U(\mathfrak{b}\oplus\C C)}\C v_{\phi,c}.$$
	
	\begin{lemm}\label{lemma633}
		Let $v_{\phi,c}\in W(\phi,c)$ be the Whittaker vector. For $v=uv_{\phi,c}\in W(\phi,c)$,
		we have
		$$(x\pm\phi(x))v=[x,u]v_{\phi,c}, \ \forall  x\in\mathfrak{b}.$$
	\end{lemm}
	\begin{proof}
		For $x\in\mathfrak{b}$, by Lemma \ref{lemma3.2} one has
		$$[x,u]v_{\phi,c}=xuv_{\phi,c}\pm uxv_{\phi,c}=xuv_{\phi,c}\pm\phi(x)uv_{\phi,c}
		=(x\pm\phi(x))uv_{\phi,c}=(x\pm\phi(x))v.$$	\end{proof}
	Using the similar calculation in  Theorem \ref{th1} and Lemma \ref{lemma633},  we conclude  that
	the induced $\T$-modules $W(\phi,c)$ are simple if $\phi(T_{\frac{1}{2}})\neq0$.
	\begin{remark}
		From the condition (i) of Theorem \ref{th1},  we see that  these  simple non $
		\mathbb{Z}_2$-graded Whittaker modules $W(\phi,c)$ for twisted $N=2$ superconformal algebras can not be obtained by our setting.
	\end{remark}

	\section*{Acknowledgements}
		H. Chen was supported by the National Natural Science Foundation of China (No. 12171129), Fujian Alliance of Mathematics (No. 2023SXLMMS05)   and Natural Science Foundation of Fujian (No. 2021J01862). Y. Su was supported by the National Natural Science Foundation of China (No. 11971350).  Y. Xiao is supported by the National Natural Science Foundation of China (No. 12401034) and the Natural Science Foundation of Shandong Province (No. ZR2023QA066). The authors would like to thank Prof. Haisheng Li for revision on first three sections.  Part of this work was done while   Chen and Xiao were visiting the Chern Institute of Mathematics, Tianjin, China. They also thank the institute and Prof. Chengming Bai for their warm hospitality and support.

	\bigskip

	Haibo Chen
	\vspace{2pt}

	Department of Mathematics, Jimei University, Xiamen, Fujian 361021, China

	\vspace{2pt}
	hypo1025@jmu.edu.cn
	
	\bigskip

	Yucai Su
	\vspace{2pt}

	Department of Mathematics, Jimei University, Xiamen, Fujian 361021, China

	\vspace{2pt}
	yucaisu@jmu.edu.cn
	
	\bigskip

	Yukun Xiao
	\vspace{2pt}
	
	1. School of Mathematics and Statistics, Qingdao University, Qingdao, 266071, China.\\
	2. Qingdao International Mathematics Center, Qingdao, 266071, China.

	\vspace{2pt}
	ykxiao@qdu.edu.cn
	
	\bigskip
	
\end{document}